\documentclass[11pt]{amsart}
\usepackage{amssymb,amscd,amsmath}
\usepackage[all]{xy}
\newtheorem{thm}{Theorem}

\newtheorem{cor}[thm]{Corollary}

\newtheorem{prop}[thm]{Proposition}

\newtheorem{lem}[thm]{Lemma}
\theoremstyle{remark}

\theoremstyle{definition}

\renewcommand{\bar}{\overline}

\newcommand{\codim}{{\rm codim}\,}

\newcommand{\Z}{{\mathbb{Z}}}

\newcommand{\N}{{\mathbb{N}}}

\newcommand{\Span}{\mathrm{Span}}

\newcommand{\Spec}{\mathrm{Spec}\;}

\newcommand{\Alt}{{\raise 2pt\hbox{$\scriptstyle\bigwedge$}}}

\newcommand{\e}{\epsilon}

\newcommand{\ba}{\mathbf{a}}
\newcommand{\bb}{\mathbf{b}}
\newcommand{\br}{\mathbf{r}}
\newcommand{\bv}{\mathbf{v}}

\newcommand{\bF}{\mathbf{F}}
\newcommand{\fI}{\mathfrak{I}}

\begin{document}
\title{Almost PI algebras are PI}

\author{Michael Larsen}
\email{mjlarsen@indiana.edu}
\address{Department of Mathematics\\
    Indiana University \\
    Bloomington, IN 47405\\
    U.S.A.}

\author{Aner Shalev}
\email{shalev@math.huji.ac.il}
\address{Einstein Institute of Mathematics\\
    Hebrew University \\
    Givat Ram, Jerusalem 91904\\
    Israel}


\thanks{ML was partially supported by NSF grant DMS-1702152.
AS was partially supported by ISF grant 686/17 and the Vinik Chair of mathematics which he holds.
Both authors were partially supported by BSF grant 2016072. AS is grateful to Efim Zelmanov for
inspiring discussions.}

\begin{abstract}
We define the notion of an almost polynomial identity of an associative algebra $R$,
and show that its existence implies the existence of an actual polynomial identity of $R$.
A similar result is also obtained for Lie algebras and Jordan algebras.
We also prove related quantitative results for simple and semisimple algebras.
\end{abstract}

\maketitle
By a well known theorem of Peter Neumann \cite{N}, for all $\epsilon > 0$ there exists $N>0$ such that if $G$ is a finite group with
at least $\epsilon|G|^2$ pairs $(x,y) \in G^2$ satisfying $[x,y]=1$, then $[x^N,y^N]^N =1$ for all $x,y\in G$.
We can express this by saying that if $[x,y]$ is an $\epsilon$-probabilistic identity for $G$, then $[x^N,y^N]^N$
is an identity for $G$.  See also Mann \cite{M} for a similar result on finite groups in which $x^2$ is an $\epsilon$-probabilistic identity.
It is an open question whether every probabilistic identity in finite and residually finite groups
implies an actual identity; see \cite{LS2} and \cite{Sh} for further discussion of this question.

In this paper, we consider the analogous problem for associative algebras (as well as Lie algebras).  Here, the possible identities are
polynomials in a non-commuting set of variables (or Lie polynomials) rather than elements of a free group.
We introduce the notion of an \emph{almost identity} of an algebra as an analogue of a probabilistic identity and show that algebras
with an almost identity satisfy an actual identity. While we focus here on infinite-dimensional algebras
we also prove related quantitative results for finite dimensional simple algebras.

Throughout this paper let $k$ denote an algebraically closed field of arbitrary characteristic.
Let $R$ be a (possibly infinite-dimensional) associative algebra over  $k$.
Let $n$ be a positive integer and $V = R^n$, regarded as $k$-vector space.
For each non-commutative polynomial $Q\in k\langle x_1,\ldots,x_n\rangle$, evaluation of $Q$ defines a map $e_Q\colon R^ n\to R$.
For each linear functional $\alpha\in R^{*}$, $\alpha\circ e_Q$ defines an element of the ring $A$ of polynomial functions on $V$.
Let $R_Q := e_Q^{-1}(0)$ denote the set of $\br = (r_1,\ldots,r_n)\in V$ such that $e_Q(\br) = 0$.

We now introduce more notation and terminology. In particular we define the notions of codimension and almost identity, which play
a key role in this paper.
Let $A$ be the ring of $k$-valued polynomial functions on a $k$-vector space $V$.  If $\fI$ is an ideal in $A$, we denote by $V(\fI)\subset V$ the solution set of the
system of equations on $V$ given by the elements of $\fI$.  An \emph{algebraic set} in $V$ is a set of this form.
Note that if $\dim V = \infty$, this may be a proper subset of the set of $k$-points of $\Spec A/\fI$.
We say the algebraic set $V(\fI)$ has \emph{finite codimension}
if it contains a translate of a vector subspace of $V$ of finite codimension.  In particular, this is the case whenever $\fI$ is finitely generated or (more generally)
when $\fI$ is contained in a finitely generated (non-unit) ideal.

We say $V(\fI)$ has codimension $\le c$ if there exists a direct sum decomposition $V = V_1\oplus V_2$, where $V_2$ is finite dimensional,
and $V$ contains $V_1\times X_2$ for $X_2\subset V_2$ an algebraic set of codimension $c$.
This is, of course, the case whenever $V(\fI)$ contains a translate of a subspace of codimension $c$.
The codimension of an algebraic set of finite codimension is the smallest integer $c$
for which the set has codimension $\le c$.
\medskip

\noindent
\emph{Example.}
Let $V$ denote the space of infinite sequences, $A$ the algebra of polynomial $k$-valued functions on $V$, $x_i\in A$  the
function sending a sequence to the value of its $i$th term, and
$$\fI = (x_1(x_1-1), x_1x_2(x_2-1),\ldots,x_1\cdots x_n(x_n-1),\ldots).$$
Then
$$V(\fI) = \{(1,1,1,\ldots)\}\cup\bigcup_{i=0}^\infty (\underbrace{1,\ldots,1}_i,0,\ast,\ast,\ldots).$$
In particular, $V(\fI)$ has codimension $1$, though it contains components of all positive integer codimensions.
\medskip

We say that \emph{$Q\neq 0$ is a $c$-almost identity of $R$} if $e_Q^{-1}(0)$ has codimension $\le c$ in $V$.
We say that $R$ \emph{satisfies an almost identity} if it satisfies a $c$-almost identity for some $c$; of course, all finite dimensional algebras have this property.
Our main result is the following:

\begin{thm}
\label{main}
For all positive integers $c$, $d$ and $n$  there exist an integer $m$ and a non-zero non-commutative polynomial $P\in k\langle y_1,\ldots,y_m\rangle$ with the following property. If $Q\in k\langle x_1,\ldots,x_n\rangle$ is a non-commutative polynomial of degree $d$
in $n$ variables, and $R$ is any associative $k$-algebra such that $Q$ is a $c$-almost identity of $R$, then
$P(r_1,\ldots,r_m)=0$ for all $r_1, \ldots , r_m \in R$.
\end{thm}

In particular, every algebra satisfying an almost identity is PI.

In the case of matrix algebras $R = M_s(k)$, if $P$ is an identity of $M_s(k)$ then $s \le \deg P/2$,
and this is tight by the Amitsur-Levitski theorem \cite{AL}.  Thus,
Theorem~\ref{main} implies
\begin{cor}
If $Q\in k\langle x_1,\ldots,x_n\rangle$ is non-zero, then the codimension of
$$M_{s,Q} := \{\br\in M_s^n\mid Q(\br)=0\}$$
in $M_s^n$ (regarded as an affine space of dimension $ns^2$) grows without bound as $s\to \infty$.
\end{cor}

In fact, something much stronger is true.  The second theorem of this paper is the following:

\begin{thm}
\label{matrix}
For each positive integer $d$
there exist real numbers $a>0$ and $b$ such that
if $Q\in k\langle x_1,\ldots,x_n\rangle$ is a non-commutative polynomial of degree $d \ge 0$,
then the codimension of
$M_{s,Q}$ in $M_s^n$ is at least $as^2+b$ for all positive integers $s$.
\end{thm}

Let $Q, d, n$ be as above and let $R$ be a finite dimensional $k$-algebra.
We define the (normalized) Hausdorff dimension of the algebraic subset $R_Q \subseteq R^n$ by $\dim R_Q / \dim R^n$.
Theorem \ref{matrix} shows that, when $R$ is simple of sufficiently large dimension given $d$, the Hausdorff dimension
of $R_Q$ is at most $1-\e$ for some $\e > 0$ depending only on $d$.

Before proving Theorems \ref{main} and \ref{matrix}, we derive some consequences and related results.

Recall that the Jacobson radical $J(R)$ of an associative ring $R$ is the
intersection of all primitive ideals of $R$.
We say that a ring $R$ is $J$-semisimple if $J(R) = 0$.

Combining Theorems \ref{main} and \ref{matrix} with other tools we obtain the following:

\begin{cor}
\label{ss}
For each $d >0$ there exist real numbers $f = f(d), g=g(d)$ such that the following holds.
Let $R$ be an associative $k$-algebra having a $c$-almost identity of degree $d$.
Suppose $R$ is $J$-semisimple. Then

(i) $R$ is a subdirect sum of matrix rings $M_{s_i}(k)$ ($i \in I$) with $s_i \le fc^{1/2}+g$ for all $i \in I$.

(ii) $R$ can be embedded in a matrix ring $M_s(C)$ for some commutative $k$-algebra $C$,
where $s \le fc^{1/2}+g$.

\end{cor}

\begin{proof}
To deduce this, note that $R$ is PI by Theorem \ref{main}. By Theorem 1 of Amitsur \cite{A}, $R$ is a
subdirect sum of central simple algebras over $k$, which are matrix rings $M_{s_i}(k)$ (since
$k$ is algebraically closed). Let $Q$ be a $c$-almost identity of $R$ of degree $d$.
Each $M_{s_i}(k)$ is a quotient of $R$, hence $Q$ is a $c$-almost identity of $M_{s_i}(k)$
for all $i \in I$. Let $a, b$ be as in Theorem \ref{matrix} above. Then this theorem yields
$as_i^2 + b \le c$ for all $i$, so $s_i \le \sqrt{(c-b)/a}$ for all $i$. This implies conclusion (i),
which in turn implies conclusion (ii) (with $C$ a direct sum of $|I|$ copies of $k$).
\end{proof}

Another consequence of Theorem \ref{matrix} deals with simple Lie algebras.

\begin{cor}
\label{lie}
For each $d>0$ there exist real numbers $a>0$ and $b$ such that
if $Q$ is a non-zero Lie polynomial of degree $d$ in $n$ variables,
$L$ is a finite dimensional simple Lie algebra over $k$ of characteristic zero,
then the codimension of $L_Q$ in $L^n$ is at least $a \cdot \dim L +b$.
\end{cor}

\begin{proof}
To deduce this, we may ignore the exceptional algebras (whose dimension is bounded)
and focus on the classical ones $L = A_r, B_r, C_r, D_r$. Any such algebra $L$ contains
the full matrix algebra $M_r(k)$ and satisfies $\dim L \le 2r^2 + r$. The Lie polynomial $Q$ can be viewed as an associative
polynomial of degree $d$. Its codimension
in $M_r(k)^n$ is a lower bound on the codimension of $Q$ in $L^n$. Applying Theorem \ref{matrix}
now completes the proof.
\end{proof}

Theorem \ref{matrix} and Corollary \ref{lie} have several applications. First, they immediately imply that
{\it if $R$ is an associative (resp. Lie) $k$-algebra defined by $n$ free generators
and by relators whose minimal degree is $d$,
and $S = M_s(k)$ (resp. a classical simple Lie $k$-algebra) of sufficiently large dimension (given $d$), then the representation variety
${\rm Hom}(R,S)$ has dimension at most $(n-\e)\dim S$, where $\e > 0$ depends only on $d$.}

Secondly, they imply the following.

\begin{prop}
\label{prob} For every $d \in \N$ there are positive real numbers $\e = \e(d), N = N(d)$ depending only on
$d$ such that the following holds. Let $F$ be a finite field, $M = M_s(F)$ a matrix algebra and $L$ a classical simple Lie algebra
over $F$.

(i) If $Q$ is a non-zero associative polynomial of degree $d$ with $n$ variables over $F$ and $\dim M \ge N$ then
\[
|M_Q| \le c|M|^{n-\e},
\]
where $c$ is a constant depending on $Q$.
Moreover, all fibers of the evaluation map $e_Q:M^n \to M$ have size at most $c|M|^{n-\e}$.

(ii) If $Q$ is a non-zero Lie polynomial of degree $d$ with $n$ variables over $F$ and $\dim L \ge N$ then
\[
|L_Q| \le c|L|^{n-\e},
\]
where $c$ is a constant depending on $Q$.
Moreover, all fibers of the evaluation map $e_Q:L^n \to L$ have size at most $c|L|^{n-\e}$.
\end{prop}

\begin{proof}
Extending scalars to $k=\bar F$ does not change the dimension of fibers of $e_Q$, so all such fibers
are bounded above by $(n-\epsilon)\dim M_s$ for part (i) and by $(n-\epsilon) \dim L$ for part (ii).
The dimensions of the polynomial equations defining any fiber are bounded above, depending only on the polynomial $Q$,
so the estimate of fiber cardinality follows from the Lefschetz trace formula and the upper bounds on Betti numbers for
affine varieties defined by polynomial equations of bounded degree \cite{Katz}.

\end{proof}

This result may be viewed as a ring-theoretic analogue of \cite[Theorem 1.2]{LS1}.

In particular it follows that if the finite simple algebras $M, L$ above satisfy a probabilistic
identity $Q$ (namely, the probability that $e_Q$ vanishes on a random $n$-tuple of elements of
$M$ or $L$ is at least some fixed $\delta > 0$), then $\dim M, \dim L$ are bounded above (in terms
of $Q$ and $\delta$).

Next we turn to almost nil algebras and algebras with an almost identity of the form $x^d$, as well as
their Lie analogues, namely almost Engel Lie algebras. Recall that Zelmanov's work on Engel Lie algebras
and related objects had remarkable group theoretic applications, such as the solution to the Restricted
Burnside Problem \cite{Z1, Z2}, as well as stronger results \cite{Z3}.

We start with finite dimensional algebras. We denote by $Rad(L)$ the solvable radical of a finite dimensional
Lie algebra $L$.

\begin{prop}
\label{nilp}

(i) Let $R$ be a finite dimensional associative $k$-algebra, and let $N$ be the subvariety
of nilpotent elements of $R$. If $\codim N = c$ then $\dim R/J(R) \le c^2$, so $R$ has a
nilpotent ideal of codimension at most $c^2$.

(ii) Let $L$ be a finite dimensional Lie $k$-algebra, where $k$ has characteristic zero, and let $N$ be
the subvariety of ad-nilpotent elements of $L$. If $\codim N = c$ then $\dim L/Rad(L) \le 4c^2 - c$.
Moreover, $L$ has a nilpotent subalgebra $K$ (consisting of ad-nilpotent elements) such that $\dim L/K\le 4c^2$.

\end{prop}

\begin{proof}
To prove part (i), write $R/J(R) = \prod_{i=1}^m M_{s_i}(k)$ and let $\pi: R \to R/J(R)$ be the canonical epimorphism.
For each $i = 1, \ldots , m$ let $N_i$ denote the set of nilpotent matrices in $M_{s_i}(k)$.
Since the elements of $J(R)$ are nilpotent we have $N = \pi^{-1}(N_1 \times \ldots \times N_m)$.
It is well known (see for instance \cite[\S1.3]{Humphreys2}) that each $N_i$ is irreducible of dimension $s_i^2 - s_i$.
Hence
\[
\sum_{i=1}^m s_i = \codim N = c.
\]
This yields
\[
\dim R/J(R) = \sum_{i=1}^m s_i^2 \le c^2.
\]

The proof of the first assertion in part (ii) is similar, using the fact that the codimension of the variety of ad-nilpotent elements
of a finite dimensional simple Lie algebra of rank $r$ is $r$. To prove the second assertion, set $R = Rad(L)$, $Z=Z(L)$.
Then $R/Z \le L/Z \cong ad L \le gl(L)$. Applying Lie Theorem (see \cite[4.1]{Humphreys1}) to the solvable Lie algebra $R/Z$
we conclude that there is
a basis for $L$ with respect to which $R/Z$ is represented by upper-triangular matrices. Let $K/Z$ denote the nilpotent
elements of $R/Z$ in its action on $L$; these are the elements represented by upper triangular matrices with zero diagonal.
Hence $K/Z$ is a Lie subalgebra of $R/Z$, and $K$ is a nilpotent Lie subalgebra of $L$ (consisting of ad-nilpotent elements).
By our assumption on $N$ the codimension of $K$ in $R$ is at most $c$. Therefore
\[
\dim L/K = \dim L/R + \dim R/K \le 4c^2.
\]
\end{proof}

The example of $R = M_c(k)$ shows that the bound in Proposition \ref{nilp} (i) is sharp;
the above argument shows that it is attained if and only if $R/J(R) = M_c(k)$.
Similarly, the bound in part (ii) of the above result is dictated by $E_8$.

Proposition \ref{nilp} (i) can be extended as follows. For each monic polynomial $P(x) \in k[x]$ of degree $s$, let
$M = M_s(k)$ and let $M_P$ be the variety of matrices in $M$ whose characteristic polynomial is $P$. Then $\dim M_P = s^2-s$.
Using arguments as above, it follows that if $P$ is a $c$-almost identity of a finite dimensional associative
$k$-algebra $R$, then $\dim R/J(R) \le c^2$.

The well known Nagata-Higman Theorem states that, if $x^d$ is an identity of an associative (non-unital) $k$-algebra $R$,
where the characteristic of $k$ is zero or greater than $d$, then $R$ is nilpotent. See \cite[Chapter 6]{DF} for this and
for explicit bounds on the degree of nilpotency of $R$ in terms of $d$. Our next result deals with rings in which
$x^d$ is an almost identity.

\begin{prop}
\label{nagata}
Let $R$ be an associative $k$-algebra, and suppose $x^d$ is a $c$-almost identity of $R$. Then

(i) $R/J(R)$ is finite dimensional; in fact $\dim R/J(R) \le c^2$.

(ii) If $R$ is a finitely generated $k$-algebra then $R$ is virtually nilpotent.
\end{prop}

Here an algebra $R$ is said to be virtually nilpotent if it has a (two-sided) nilpotent ideal of finite codimension.

\begin{proof}
To prove Proposition \ref{nagata}, apply Corollary \ref{ss} to deduce that $R/J(R)$ is a subdirect sum
of matrix rings $M_{s_i}(k)$ ($i \in I$). In particular, $R/J(R)$ is
residually finite dimensional, so it suffices to show that all its finite dimensional
quotients have dimension at most $c^2$. Let $S$ be such a quotient. Then $J(S)=0$ and the codimension
of the solutions of $x^d=0$ in $S$ is at most $c$. In particular, the codimension of the nilpotent elements
of $S$ is at most $c$. Applying Proposition \ref{nilp} (i) to $S$ we obtain
$\dim S = \dim S/J(S) \le c^2$, proving part (i).

To prove part (ii) we apply the Kemer-Braun Theorem on the nilpotency of $J(R)$ for finitely generated PI algebras $R$ \cite{B2}.
Since $R$ is PI by Theorem \ref{main} above, part (ii) now follows from part (i).
\end{proof}

\bigskip

We now turn to the proofs of the main results, namely Theorems \ref{main} and \ref{matrix}.

\begin{lem}
\label{density}
If $\fI\subset A$ is a non-zero ideal and $\pi\colon V \to W$ is a $k$-linear map onto a finite dimensional vector space,
then every element of $\pi(V)$ can be expressed as the sum of two elements of $\pi(V\setminus V(\fI))$.
\end{lem}

\begin{proof}
As $\fI$ is non-zero, it contains $(f)$ for some non-zero polynomial $f\in A$, so it suffices to prove the theorem when $\fI = (f)$.
We choose a surjective linear map $\phi\colon V\to W_0$, for some finite-dimensional $W_0$, such that $f$ factors through $\phi$.
For finite dimensional vector spaces, the image of a Zariski-dense subset under any $k$-linear surjective map is again Zariski-dense.
Replacing $W$ by $W\oplus W_0$ and $\pi$ by $(\pi,\phi)$,
we may assume without loss of generality that $f$ comes by composition by $\pi$ from a well-defined non-zero polynomial on $W$.  Thus, $\pi(V\setminus V(\fI))$
is the complement of the zero locus of a non-zero polynomial.  By the irreducibility of affine spaces,  every non-empty Zariski-open subset $U$ of a vector space $\pi(V)$ has the property
that $U+U = \pi(V)$.
\end{proof}

We fix positive integers $n$ and $d$.
Let $M^d$ denote the set of
ordered $d$-tuples of elements (or, equivalently, $d$-term sequences) in $\{1,2,\ldots,n\}$.  We identify $I = (i_1,\ldots,i_d)\in M^d$
with $x_{i_1}\cdots x_{i_d}\in k\langle x_1,\ldots,x_n\rangle$.
For $\br = (r_1,\ldots,r_n)\in R^n$, we define
$$\br_I := r_{i_1}\cdots r_{i_d}.$$
Let
$$M^{\le d} = \bigcup_{i=0}^d M^i,$$
and let $k\langle x_1,\ldots,x_n\rangle^{\le d}$ denote its linear span.

\begin{prop}
\label{indep}
Given $n$ and $d$ as above, there exists a non-zero homogeneous non-commutative polynomial $P\in k\langle x_1,\ldots,x_n\rangle$ such that if $R$ is an associative $k$-algebra
and $\br\in R^n$ satisfies $P(\br) \neq 0$, then
$\{\br_I\mid I\in M^{\le d}\}$ is linearly independent.
\end{prop}

To prove this we need some preparations.
For $I\in M^d$ and any positive integer $N$, we let $S_{N,I}$ denote the set of distinct $N$-term sequences in $\{1,\ldots,n\}$ for which $I$ is not a consecutive subsequence.
We identity $S_{N,I}$ with a subset of the monomials in $k\langle x_1,\ldots,x_n\rangle$.

\begin{lem}
\label{pattern}
With notation as above,
$$\limsup_N |S_{N,I}|^{1/N} < n.$$
\end{lem}

\begin{proof}
Write $N = qd + r$ where $q, r \in \N$ and $0 \le r < d$. Avoiding the subsequence $I$ in locations $td+1, \ldots , (t+1)d \le N$
($t \in \N$) yields $|S_{N,I}| \le (n^d-1)^q n^r \le (n^d-1)^{N/d} n^d$. Therefore
\[
|S_{N,I}|^{1/N} \le (n^d-1)^{1/d} n^{d/N},
\]
so $\limsup_N |S_{N,I}|^{1/N} \le (n^d-1)^{1/d} < n$.

\end{proof}

Given an $n$-tuple of functions $\bF := (F_1(x),\ldots,F_n(x))$, $F_i\colon \Z^{>0}\to \Z^{>0}$, and a sequence $\ba := a_1,a_2,\ldots,a_N\in \{1,\ldots,n\}$, we define the
\emph{score} of $\ba$ with respect to $\bF$ to be
$$\sigma(\ba) := \sum_{j=1}^N F_{a_j}(j).$$
Given $d$ and $n$, we construct $\bF$ as follows.  We fix a prime $p > 2^{dn}$.  For $1\le i\le n$ and $j>0$, we let $r$ denote the remainder when $jn+i$ is divided by $dn$ and define
$F_i(j)$ to be the unique integer congruent to $2^r$ (mod $p$) for which
$$jp \le F_i(j) < (j+1)p.$$
Thus, for any sequence $\ba$ of length $N$,
\begin{equation}
\label{bounds}
\frac{N(N+1)}2p\le \sigma(\ba) < \frac{(N+1)(N+2)-2}2p.
\end{equation}
This implies that if the sequence $\ba$ is longer than the sequence $\bb$, then $\sigma(\ba) > \sigma(\bb)$.

\begin{lem}
\label{distinct}
For all integers $t$ and $N$ with $1\le t\le N-d+1$,
if $\ba$ and $\bb$
are distinct $N$-term sequences as above, and $a_j= b_j$ except if $t\le j < t+d$,
then $\sigma(\ba) \neq \sigma(\bb)$.
\end{lem}

\begin{proof}
It suffices to prove that
$$\sum_{j=t}^{t+d-1} F_{a_j}(j) \neq \sum_{j=t}^{t+d-1} F_{b_j}(j).$$
We show that the two sides are not congruent to one another (mod $p$).
Each is congruent (mod $p$) to a sum of $d$ elements of the set $\{1,2,4,8,\ldots,2^{dn-1}\}$.
As $p > 2^{dn}$, two such sums are congruent (mod $p$) if and only if the sums of
powers of $2$ are the same, i.e., if and only if the set of remainders obtained when  $jn+a_j$ is divided by $dn$
is the same as the set of remainders obtained when $jn+b_j$ is divided by $dn$ as $j$ ranges over $t,\ldots,t+d-1$.
As $1\le a_j,b_j\le n$, this is equivalent to the condition that $a_j=b_j$ for all $j$ in this range, contrary to hypothesis.
\end{proof}

We now prove Proposition~\ref{indep}.

\begin{proof}
By induction on $d$, we may assume that $\br$ satisfies some non-commutative polynomial equation of degree exactly $d$, i.e., that
for some $I\in M^d$,
$\br_I$ can be expressed
\begin{equation}
\label{relation}
\br_I = \sum_{J\in M^{\le d}\setminus \{I\}} c_J\br_J.
\end{equation}
By Lemma~\ref{distinct}, without loss of generality, we may assume that $c_J \neq 0$ implies $\sigma(J) < \sigma(I)$.

We introduce formal (commuting) variables $z_J$ indexed by $J\in M^{\le d}\setminus \{I\}$.
Let $B$ denote the (commutative) polynomial algebra over $k$ generated by the variables  $z_J$ and consider the free associative $B$-algebra
$B\langle x_1,\ldots,x_n\rangle$.  Let $\fI$ denote the two-sided ideal in this algebra generated by
$$x_I - \sum_{\{J\mid c_J \neq 0\}} z_J x_J.$$
Let $B\langle x_1,\ldots,x_n\rangle_{\le N}$ denote the $k$-subspace of
$B\langle x_1,\ldots,x_n\rangle$ consisting of monomials in $S_{M,I}$ for $M\le N$ multiplied by polynomials in the $z_J$ of degree
at most the maximal score of a sequence of length $\le N$.

We claim that every element $\alpha$ of $k\langle x_1,\ldots,x_n\rangle^{\le N}$ is congruent (mod $\fI$)
to an element $\beta$ of $B\langle x_1,\ldots,x_n\rangle_{\le N}$.
We iteratively construct a sequence of elements of $B\langle x_1,\ldots,x_n\rangle$ which lie in $\alpha+\fI$.
Each monomial $x_{j_1}\cdots x_{j_M}$ which is not in $S_{M,I}$
can be replaced by a linear combination of monomials
associated to sequences of lower score than $j_1,\ldots,j_M$, with coefficients of the form $z_J$.
The number of steps before this process terminates is at most the maximum score of a monomial of degree $N$,
which, by (\ref{bounds}), is less than $p(N^2+3N)/2$.
The resulting element, $\beta$,
lies in  $B\langle x_1,\ldots,x_n\rangle_{\le N}$.  The space of polynomials in $|M^{\le d}|-1$ variables of degree less than $p(N^2+3N)/2$
has dimension at most $C N^{2(|M^{\le d}|-2)}$ for some $C$ not depending on $N$.

For any $\epsilon > 0$, we have
\begin{equation}
\label{ineq}
\dim k\langle x_1,\ldots,x_n\rangle^{\le N} > n^N > C N^{2(|M^{\le d}|-2)} (n-\epsilon)^N.
\end{equation}
By Lemma~\ref{pattern}, when $N$ is sufficiently large,
$$\dim k\langle x_1,\ldots,x_n\rangle^{\le N} > \dim B\langle x_1,\ldots,x_n\rangle_{\le N},$$
and it follows that there exists a non-zero $\alpha$ which is equivalent to $\beta = 0$
(mod $\fI$).

Substituting $c_J$ for each $z_J$ and $r_i$ for each $x_i$ in $\alpha$, by (\ref{relation}), we get $0$.  Defining $P_I$ to be $\alpha$
and defining $P$ to be the product of $P_I$ over all $I\in M^d$, we obtain a polynomial which vanishes on any $\br$ for which $Q(\br)=0$ for any
non-commutative polynomial $Q$ of degree $d$.

\end{proof}

\begin{lem}
\label{subspace}
Let $r_1,\ldots,r_n$ denote elements of an associative algebra $R$ such that the
monomials of degree $\le d$ in the $r_i$ are linearly independent.  If $s_1,\ldots,s_m$
are linearly independent  in $\Span_k(r_1,\ldots,r_n)$, then the monomials of degree $\le d$ in the $s_j$ are linearly independent.
\end{lem}

\begin{proof}
Let $V$ denote the span of the $r_i$ and $W$ the span of the $s_j$.  The linear independence of the degree $d$ monomials in the $r_i$ (resp. $s_j$) is equivalent to the injectivity of the product map $\bigoplus_{i=0}^dV^{\otimes i}\to R$ (resp. $\bigoplus_{i=0}^dW^{\otimes i}\to R$), so the lemma follows from the fact that the latter map factors through the former and the map
$$\bigoplus_{i=0}^dW^{\otimes i}\to \bigoplus_{i=0}^dV^{\otimes i}$$
is injective.
\end{proof}

\begin{lem}
\label{dim}
If $n<m$ are positive integers, the algebraic set $N_{n,m}$ consisting of
of $n\times m$ matrices of rank strictly less than $n$ has codimension at least $1+m-n$
in $M_{n\times m}(k) = k^{mn}$.
\end{lem}

\begin{proof}
If $Z_{n,m}$ consists of ordered pairs $(X,v)$ consisting of an $n\times m$ matrix and a column vector of size $n$ such that $Xv=0$, then $Z_{n,m}$ projects onto the $n$-dimensional space of column vectors $v$, and the dimension of every fiber except the $0$-fiber is $mn-m$, so $\dim Z_{n,m} = mn+n-m$.
Projecting onto the first factor, the non-empty fibers have dimension at least $1$, so the dimension of the image is at most $mn+n-m-1$.  Its codimension is therefore at least $1+m-n$.
\end{proof}

In fact, the codimension is exactly $1+m-n$; this follows immediately from \cite[Exercise 10.10]{Eisenbud}.

We can now prove Theorem~\ref{main}.

\begin{proof}
Assume that $Q$ is a $c$-almost identity of $R$.  We fix a $k$-linear direct sum decomposition $R^n = V =  V_1\oplus V_2$ where $V_2$ is finite dimensional and contains an
algebraic set $X_2$ of codimension $c$ such that $Q$ vanishes on $V_1\times X_2$.
We fix $m := n+\max(n,c)$.  By Proposition~\ref{indep}, there exists a non-zero element $P\in k\langle y_1,\ldots,y_m\rangle$ such that $P(r_1,\ldots,r_m)\neq 0$
implies that the monomials in $r_i$ of degree $\le d$ are linearly independent.

We assume that $P$ is not identically zero on $R$.
Applying Lemma~\ref{density} to $(P)$, we may choose the $r_i$ such that $(r_1+r_2,r_3+r_4,\ldots,r_{2n-1}+r_{2n})\in V_1\times X_2$.

Right-multiplication by the column vector $(r_1\,r_2\,\cdots\,r_m)^t$ embeds $M_{n\times m}(k)$ linearly in $V$, and we have already seen that
the image contains at least one point of $V_1\times X_2$.
By Lemma~\ref{subspace}, every matrix of rank $n$ in $M_{n\times m}(k)$ maps to an element of $R^n$ which does not satisfy any non-commutative polynomial equation of degree $\le d$
and in particular does not lie in $V_1\times X_2$.  By Lemma~\ref{dim}, the intersection of $M_{n\times m}(k)$ with $V_1\times X_2$ has codimension $>c$.
However, in a finite dimensional vector space, a non-empty intersection of a subvariety of codimension $c$ with any irreducible subvariety is of codimension $\le c$ in the latter.

Thus $P$ is a polynomial identity of $R$.
\end{proof}

We remark that the analogue of Theorem~\ref{main} holds for non-unital associative algebras as well.  The proof is the same
except that polynomials cannot have a constant term.

There is also an analogue for Lie algebras.  Namely, given positive integers $c,d,n$, there exists a Lie polynomial $P$ such that for every Lie polynomial $Q$ of degree $d$ in $n$ variables and every Lie algebra $L$, either $P$ is an identity on $L$, or the variety defined by $Q$ on $L^n$ has codimension $>c$.  The proof is the same except that
$k\langle x_1,\ldots,x_n\rangle^{\le N}$ must be replaced by the subspace of Lie polynomials in $x_1,\ldots,x_n$ of degree $\le N$,
and the left hand side of (\ref{ineq}) must be replaced by the dimension of this subspace.  By \cite[II, \S3, Th\'eor\`eme~2]{Bourbaki}, it is
$$\sum_{M=1}^N \frac 1M \sum_{d\mid M} \mu(d) n^{M/d}
\ge \frac 1N\sum_{d\mid N} \mu(d) n^{N/d}
> \frac{n^N - \sum_{i\le N/2} n^i}N > \frac{n^N - 2n^{N/2}}N$$
It follows that for fixed $n$ and $N$ sufficiently large,
$$\frac 1N \sum_{d\mid N} \mu(d) n^{N/d} > \binom{n+N-1}N(n-1)^N.$$

The analogue of Theorem~\ref{main} holds also for Jordan algebras.  It suffices to show that for some $\epsilon < 1$ the
dimension of the space of Jordan polynomials in $x_1,\ldots,x_n$ of degree $\le N$ is
greater than
than $(n-\epsilon)^N$ for some $N$.  By \cite[Theorem~9]{Robbins}, if $c_{n,i}$ is the dimension of
the part of the free Jordan algebra on $n$ generators homogeneous of degree $i$, then
$$\sum_{i=0}^\infty c_{n,i}t^i = \exp \sum_{i=1}^\infty \frac{n^{\theta(i)} t^i}i,$$
where $\theta(i)$ denotes the largest odd divisor of $i$.  As $\sum_{i=1}^\infty \frac{n^{\theta(i)} t^i}i$
grows without bound as $t$ approaches $1/n$ from below, it follows that the radius of convergence of $\sum_{i=0}^\infty c_{n,i}t^i$
is at most $1/n$, from which we deduce that
$$\limsup \frac{\log c_{n,i}}i \ge n,$$
which implies the needed analogue of (\ref{ineq}).
\medskip

The proof of Theorem~\ref{matrix} follows the method used to bound the size of fibers of word maps for groups of finite simple groups of Lie type in \cite{LS1}, except
that since we are working over an algebraically closed field instead of a finite field, we use dimension theory as a substitute for the counting arguments in that paper.

\begin{proof}

Let $N$ denote the cardinality of $M^{\le d}$.
We order $M^{\le d}$ by increasing degree and within each degree, lexigraphically, denoting its elements,  in increasing order, $I_1, I_2, \cdots, I_N$.

For $\br = (r_1,\ldots,r_n)\in M_s(k)^n$, and $\bv = (v_1,\ldots,v_t)\in (k^s)^t$, we consider the sequence
\begin{equation}
\label{seq}
\br_{I_1} v_1,\ldots,\br_{I_N} v_1, \br_{I_1} v_2,\ldots,\br_{I_N} v_2,\ldots,\br_{I_N} v_t.
\end{equation}
If $e_Q(\br)=0$ for some non-zero $Q\in k\langle x_1,\ldots,x_m\rangle^{\le d}$,
then $e_Q(\br) v_i = 0$ for $1\le i\le t$, so for each $i$ in this range, there exists $j_i\le N$ such that
$r_{I_{j_i}}v_i$ is a linear combination of previous terms in the sequence (\ref{seq}).
Let $\Sigma_{n,d,t}$ denote the algebraic set consisting of pairs $(\br,\bv)\in M_s(k)^n\times (k^s)^t$
such that this condition holds.  We claim
\begin{equation}
\label{dim-est}
\dim \Sigma_{n,d,t} \le ns^2+Nt^2.
\end{equation}
Since $M_{s,Q}\times (k^s)^t\subset \Sigma_{n,d,t}$, this implies
$$\dim M_{s,Q} \le ns^2+Nt^2-st.$$
Setting $t = \lfloor s/2N\rfloor$, we obtain
$$\dim M_{s,Q} \le (n-1/4N)s^2+s,$$
which implies the theorem.

To prove (\ref{dim-est}), we consider a fixed $(\br,\bv)\in \Sigma_{n,d,t}$ and define $j_i$, for $1\le i\le t$,
to have the smallest value for which $r_{I_{j_i}}v_i$ is a linear combination of previous terms in the sequence (\ref{seq}).
It suffices to prove that for each $i$, if we fix all the terms preceding $r_{I_{j_i}}v_i$, the condition on $(\br,\bv)$ that
$r_{I_{j_i}}v_i$ lies in the span of these fixed vectors is a linear condition of codimension greater than $s-Nt$, and moreover, the conditions imposed
by successive values of $i$ are independent of one another.  If $j_i=1$, since $I_1$ is the tuple of length $0$,
the linear dependence condition requires $v_i$ to lie in the span of the $(i-1)N$ preceding terms of the sequence (\ref{seq}).
This is a condition of codimension greater than $s-Nt$, and as it is the only condition we consider which applies to $v_i$,
it is linearly independent of our other conditions.  If $j_i > 1$, then $x_{I_{j_i}}$ can be written $x_{p_i} x_{J_i}$ for some $1\le p_i\le n$
and some $J_i < I_{j_i}$.  The linear dependence condition can now be viewed as a condition on $r_{p_i}$, and it asserts that
$r_{p_i} (\br_{J_i} v_i)$ lies in a specified vector space of dimension less than $Nt$.  Note that $\br_{J_i} v_i$ belongs to the part of
the sequence (\ref{seq}) that we are assuming fixed, so this is a linear condition on $r_{p_i}$ of codimension $> s - Nt$.
Note also that by definition of $j_i$, $\br_{J_i} v_i$ is linearly independent from all previous terms in the sequence (\ref{seq}),
so the condition on $r_{p_i}$ is linearly independent of any previous conditions on $r_{p_i}$.

This concludes the proof of (\ref{dim-est}) and therefore of Theorem~\ref{matrix}.

\end{proof}


\begin{thebibliography}{CRS}

\bibitem[A]{A} Amitsur, Shimshon: An embedding of PI-rings, {\it Proc. Amer. Math. Soc.} {\bf 3} (1952), 3-–9.

\bibitem[AL]{AL} Amitsur, Shimshon; Levitzki, Jacob: Minimal identities for algebras, {\it Proc. Amer. Math. Soc.}
{\bf 1} (1950), 449–-463.

\bibitem[B]{Bourbaki}Bourbaki, N.:
\'El\'ements de math\'ematique. Fasc. XXXVII. Groupes et alg\`ebres de Lie. Chapitre II: Alg\`ebres de Lie libres. Chapitre III: Groupes de Lie. Actualit\'es Scientifiques et Industrielles, No. 1349. Hermann, Paris, 1972.

\bibitem[B2]{B2} Braun, Amiram: The nilpotency of the radical in a finitely generated PI ring, {\it J. Algebra} {\bf 89} (1984), 375-–396.


\bibitem[CRS]{CRS}
Cvetkovi\'c, Drago\v{s}; Rowlinson, Peter; Simi\'c, Slobodan:
An introduction to the theory of graph spectra.
London Mathematical Society Student Texts, 75. Cambridge University Press, Cambridge, 2010.

\bibitem[DF]{DF} Drensky, Vesselin; Formanek, Edward: Polynomial identity rings. Advanced Courses in Mathematics, CRM Barcelona,
Birkh{\" a}user Verlag, Basel, 2004.

\bibitem[E]{Eisenbud} Eisenbud, David:
Commutative algebra. With a view toward algebraic geometry. Graduate Texts in Mathematics, 150. Springer-Verlag, New York, 1995.

\bibitem[H1]{Humphreys1} Humphreys, James E.:
Introduction to Lie algebras and their representations. Graduate texts in Mathematics, 9. Springer-Verlag, New York, 1972.


\bibitem[H2]{Humphreys2} Humphreys, James E.:
Conjugacy classes in semisimple algebraic groups. Mathematical Surveys and Monographs, 43. American Mathematical Society, Providence, RI, 1995.

\bibitem[K]{Katz} Katz, Nicholas M.:
Sums of Betti numbers in arbitrary characteristic,
{\it Finite Fields Appl.} {\bf 7} (2001), 29--44.

\bibitem[LW]{LW} Lang, S.; Weil A.: Number of points of varieties over finite
fields, {\it Amer. J. Math.} {\bf 76} (1954), 819--827.

\bibitem[LS1]{LS1} Larsen, Michael; Shalev, Aner:
Fibers of word maps and some applications,  {\it J.\ Algebra} {\bf 354} (2012), 36--48.

\bibitem[LS2]{LS2} Larsen, Michael; Shalev, Aner: A probabilistic Tits alternative
and probabilistic identities,
{\it Algebra and Number Theory} {\bf 10} (2016), 1359--1371.

\bibitem[M]{M} Mann, Avinoam:
Finite groups containing many involutions,
{\it Proc.\ Amer.\ Math.\ Soc.} {\bf 122} (1994), 383--385.

\bibitem[N]{N} Neumann, Peter M.: Two combinatorial problems in
group theory, {\it Bull. London Math. Soc.} {\bf 21} (1989),
456--458.

\bibitem[R]{Robbins}Robbins, David P.:
Jordan elements in a free associative algebra, I,
{\it J.\ of Algebra} {\bf 19} (1971), 354--378.

\bibitem[Sh]{Sh} Shalev, Aner: Probabilistically nilpotent groups, {\it Proc. Amer. Math. Soc.}
{\bf 146} (2018), 1529-–1536.


\bibitem[Z1]{Z1} Zelmanov, Efim: Solution of the restricted Burnside problem for groups of odd exponent,
(Russian) \emph{Izv. Akad. SSSR Ser. Mat.} \textbf{54} (1990), 42-–59, 1990; translation in \emph{Math
USSR-Izv.} \textbf{36} (1990), 41–-60.

\bibitem[Z2]{Z2} Zelmanov, Efim: Solution of the restricted Burnside problem for 2-groups, (Russian)
\emph{Mat. Sb.} \textbf{182} (1991), 568-–592; translation in \emph{Math USSR-Sb} \textbf{72} (1992), 543-–565.

\bibitem[Z3]{Z3} Zelmanov, Efim: Lie algebras and torsion groups with identity, {\it J. Comb. Algebra} {\bf 1} (2017), 289--340.

\end{thebibliography}
\end{document}